\documentclass[10 pt]{article}
\usepackage{amsthm}
\usepackage{graphicx, amssymb, amstext, amsmath, pdfsync, tikz, amscd}
\usepackage{xcolor}

\usepackage{tikz}
\usepackage{tikz}
\usetikzlibrary{calc}
\usetikzlibrary{matrix,arrows}
\usepackage{verbatim}
\usepackage{tocloft}
\usepackage{setspace}
\usepackage{mathtools}
\usepackage{authblk}
\usepackage{multirow}
\usepackage[all]{xy}

\usepackage{enumitem}    
\setlist[enumerate,2]{
  ref=(\alph*),
}

\usepackage{mathrsfs}

\usepackage[round, sort, numbers]{natbib}
\setcitestyle{square}

\usepackage[english]{babel}
\usepackage[none]{hyphenat}
\usepackage[tmargin=1in, bmargin=1in, lmargin=1in, rmargin=1in]{geometry}
\usetikzlibrary{arrows}

\usepackage{color}
\usepackage[colorlinks, pagebackref]{hyperref}
\hypersetup{
  colorlinks=true,
  citecolor=red,
  linkcolor=blue,
  urlcolor=blue}

\newtheorem{theorem}{Theorem}[section]
\newtheorem{lemma}[theorem]{Lemma}
\newtheorem{proposition}[theorem]{Proposition}
\newtheorem{corollary}[theorem]{Corollary}
\newtheorem{conjecture}[theorem]{Conjecture}

\theoremstyle{definition}
\newtheorem{definition}[theorem]{Definition}

\newtheorem{remark}[theorem]{Remark}

\newtheorem{notation}[theorem]{Notation}

\newtheorem{question}[theorem]{Question}
\DeclareMathOperator{\Hom}{Hom}

\newcommand{\colim}{\operatorname{colim}}
\newcommand{\Spec}{\operatorname{Spec}}
\newcommand{\Proj}{\operatorname{Proj}}

\newcommand{\A}{{\mathbb A}}

\newcommand{\pone}{{\mathbb P}^1}

\newcommand{\Nis}{\operatorname{Nis}}

\newcommand{\Sm}{\operatorname{Sm}}

\newcommand{\Shv}{{\operatorname{Shv}}}
\newcommand{\PShv}{{\operatorname{PShv}}}

\newcommand{\<}[1]{{\langle}#1 {\rangle}}

\newcommand{\simpnis}{{\Delta}^{\circ}\Shv_{\Nis}({\Sm}/k)}

\newcommand{\ie}{{\it i.e.\/},\ }

\newcommand{\calH}{\mathcal H}

\def\<{\langle}
\def\>{\rangle} 
\def\-{\overline} 
\def\~{\widetilde}
\def\^{\widehat}

\def\@{\mathcal}
\def\!{\mathscr}
\def\&{\mathbf}
\def\_{\underline}
 
\def\x{\times}

\providecommand{\keywords}[1]
{
  \small    
  \textbf{\textbf{Keywords---}} #1
}

\begin{document}
\title{$\A^1$-connected components of blow-up of threefolds fibered over a surface}
\author{\small Rakesh Pawar}
\affil{\small Department of Mathematics\\
Indian Institute of Science Education and Research (IISER), Mohali\\
Email: rakesh.pawar.math@gmail.com
}
  
\date{}
\maketitle
\begin{abstract}
Over a perfect field, we determine the sheaf of $\A^1$-connected components of a class of threefolds given by the Blow-up of a variety admitting a $\pone$-fibration over either an $\A^1$-rigid or a non-uniruled surface, along a smooth curve. As a consequence, we verify that the sheaf of $\A^1$-connected components for such varieties is $\A^1$-invariant.  
\end{abstract}




\section{Introduction}

F. Morel and V. Voevodsky constructed in~\cite{mv99} the motivic homotopy category suitable for schemes, with the affine line $\A^1$ playing the role of the unit interval in the classical homotopy theory. Briefly speaking, one enlarges the category $\Sm/k$ of smooth schemes over a field $k$, to the category $\simpnis$ of simplicial  Nisnevich sheaves of sets over the Nisnevich site $(\Sm/k)_{\Nis}$ of smooth schemes over the field. By localizing the category with the morphisms given by $\A^1$-weak equivalences, we get the $\A^1$-homotopy category. As in the classical homotopy theory, there is the analogous notion of the Nisnevich sheaf $\pi_0^{\A^1}(\mathcal{X})$ of $\A^1$-connected components of a space $\mathcal{X}$ and the Nisnevich sheaves of $\A^1$-homotopy groups $\pi_i^{\A^1}(\mathcal{X}, x)$ for $i\geq 1$ and a pointed space $(\mathcal{X},x)$. F. Morel showed that the higher homotopy groups
$\pi_i^{\A^1}(\mathcal{X}, x)$ 
for $i\geq 1$ are $\A^1$-invariant. Consequently following conjecture was proposed by F. Morel. 
\begin{conjecture}(\cite[Conjecture 1.12]{morel})
\label{morelconj}
 For any simplicial sheaf $\mathcal{X}$, the sheaf $\pi_0^{\A^1}(\mathcal{X})$ is $\A^1$-invariant. \end{conjecture}
 Morel's conjecture holds for $\A^1$-rigid schemes (in the sense of the Remark~\ref{a1rigid}), in particular for smooth projective curves of genus $>0$, abelian varieties. More non-trivial cases for which Morel's conjecture is verified include H-spaces, homogenous spaces for H-groups \cite{ch14}.
 Recently it has been verified for non-uniruled smooth projective surfaces over any perfect field in \cite{bhs} and smooth projective birationally ruled surfaces over an algebraically closed field of characteristic 0 in \cite{bs19}. Using the classification of smooth projective surfaces over an algebraically closed field of characteristic 0, Morel's conjecture holds for surfaces. In this paper, we are interested in \textit{the universal $\A^1$-invariant quotient} (\ref{uniquotient}) of a particular class of smooth projective threefolds.  As a consequence we show that, conjecture~\ref{morelconj} holds for this class of threefolds.
 
 Given a scheme $X$ over the field $k$, we consider the \textit{Morel-Voevodsky singular construction} $Sing_*^{\A^1}(X)$ for $X$. We define the sheaf of $\A^1$-chain connected components of $X$, denoted by $\mathcal{S}(X)$ to be the Nisnevich sheafification of the presheaf on $\Sm/k$ given by
\begin{center}
$V\mapsto \pi_0(Sing_*^{\A^1}X(V))$
\end{center}
for $V\in \Sm/k$. 

By iterating this construction, for a natural number $n>0$, one gets a sequence of sheaves  $\{\mathcal{S}^n(X)\}_{n\geq 0}$ with epimorphisms  $\mathcal{S}^n(X)\to  \mathcal{S}^{n+1}(X)$. We consider \textit{the universal $\A^1$-invariant quotient} $$\mathcal{L}(X):=\colim_{n\geq 0}  \mathcal{S}^n(X).$$ The canonical map  $\pi_0^{\A^1}(\mathcal{X})\to \mathcal{L}(\mathcal{X}) $ is an isomorphism if and only if the sheaf $\pi_0^{\A^1}(\mathcal{X})$ is $\A^1$-invariant as shown in~\cite[Theorem 2.16]{bhs}.
As a step towards Morel's conjecture, the natural question is
  \begin{question}{\cite[Remark 3.17]{bhs}}
 Given a scheme $X$ over a field $k$, does there exist $n$, such that the sequence $\{\mathcal{S}^n(X)\}_{n\geq 0}$ of sheaves stabilizes i.e $\mathcal{S}^n(X)\to \mathcal{S}^{n+1}(X)$ is an isomorphism for some $n>0$ (possibly depending on $X$)?  
 \end{question}
 In general, the sheaf of $\A^1$-chain connected components $\mathcal{S}(X)$ is not isomorphic to $\pi_0^{\A^1}(X)$ as shown in \cite[Section 4]{bhs}.
It was shown in \cite{bs15}, that for $\A^1$-connected anisotropic algebraic groups over an infinite perfect field, the sheaf $\mathcal{S}(X)$ is not isomorphic to $\pi_0^{\A^1}(X)$, 
but rather a further iteration $\mathcal{S}^n(X)$ is isomorphic to $\pi_0^{\A^1}(X)$.

In this paper, we prove Proposition~\ref{n-ghostlyingovergamma} regarding ghost homotopies for smooth proper varieties admitting a $\pone$-fibration over a surface. In particular, we consider smooth proper varieties admitting a $\pone$-fibration over either a non-uniruled surface or an $\A^1$-rigid surface. We prove the following results in this paper.

\begin{theorem}[Theorem~\ref{nonunisurface}]
\label{MT1}
Let $B$ be a non-uniruled, smooth, proper variety of dimension 2 over a perfect field $k$. Let $X$ be a smooth projective over $k$ of dimension 3 which admits a morphism $\pi: X\to B$ of schemes over $k$ such that $\pi$ is a $\mathbb{P}^1$-fibration. Let $W$ be a smooth irreducible curve in $B$. Let $Z$ be a smooth closed subscheme of $X$ such that $Z$ is a section 
of the $\pone$-fibration on $W$, pulled back from $X$.
 Let $\widetilde{X}$ be the blow up of $X$ along the closed subscheme $Z$.
 Then 
 $\mathcal{S}^2(\widetilde{X})\simeq\mathcal{S}^3(\widetilde{X}).$
\end{theorem}
As a consequence, we deduce the following
\begin{corollary}[Corollary~\ref{cor1}]
For $\widetilde{X}$ satisfying the assumptions of Theorem~\ref{MT1}, $\pi_0^{\A^1}(\widetilde{X})$ is $\A^1$-invariant.
\end{corollary}

\begin{theorem}[Theorem~\ref{a1rigidsurface}]
\label{MT2}
Let $B$ be an $\A^1$-rigid, smooth, proper variety of dimension 2 over a perfect field $k$. Let $X$ be a smooth projective over $k$ of dimension 3 which admits a morphism $\pi: X\to B$ of schemes over $k$ such that $\pi$ is a $\mathbb{P}^1$-fibration. Let $W$ be a smooth irreducible curve in $B$. Let $Z$ be a smooth closed subscheme of $X$ such that $Z$ is a section 
of the $\pone$-fibration on $W$, pulled back from $X$.
 Let $\widetilde{X}$ be the blow up of $X$ along the closed subscheme $Z$. 
  Then 
 $\mathcal{S}^2(\widetilde{X})\simeq\mathcal{S}^3(\widetilde{X}).$
\end{theorem}

As a consequence, we deduce the following
\begin{corollary}[Corollary~\ref{cor2}]
For $\widetilde{X}$ satisfying the assumptions of Theorem~\ref{MT2}, $\pi_0^{\A^1}(\widetilde{X})$ is $\A^1$-invariant.

\end{corollary}
We also provide in Remark~\ref{sing*}, examples  of smooth projective varieties $\widetilde{X}$ of dimension 3 such that $Sing_*^{\A^1}(\widetilde{X})$ is not $\A^1$-local. 

\paragraph{Structure of the paper} In Section~\ref{S2}, we recall some definitions and basic facts on Ghost homotopies from the literature relevant to this paper. In Section~\ref{S3} we prove Theorem~\ref{a1rigidsurface} and Theorem~\ref{nonunisurface}, which are the main results of the paper. The Proposition~\ref{n-ghostlyingovergamma} is key step in the proofs of the main results and its proof is deferred to the Section~\ref{S4}.  
\begin{notation}
Throughout the paper, $k$ will denote a fixed base field. By a variety over $k$ we mean an integral, separated, finite type scheme over $\Spec k$. We refer to \cite{har} for basic definitions and facts in Algebraic Geometry. 
\end{notation}




\section{Preliminaries} 
\label{S2}
In this section, we recall some basic definitions and standard facts relevant to our discussion.

Let $\Sm/k$ denote the big site of smooth, finite type, separated schemes over $\Spec k$ with the Nisnevich topology. 
Let $\Delta^{op}(\PShv(\Sm/k))$ ( and $\Delta^{op}(\Shv(\Sm/k))$) denote the category of simplicial presheaves (resp. sheaves) of sets. We consider all presheaves of sets on $\Sm/k$ as constant simplicial presheaves.
By inverting the weak eqivalences on the simplicial model category $\Delta^{op}(\Shv(\Sm/k))$, 
we get the homotopy category $\mathcal{H}_s(k)$.

By the technique of Bousfield localization with respect to the collection of the projection morphisms given by $\mathcal{X}\times\A^1\to \mathcal{X}$ for the objects $\mathcal{X}$ in  $\Delta^{op}(\Shv(\Sm/k))$, one gets $\A^1$-local model structure on  $\Delta^{op}(\Shv(\Sm/k))$ with $\A^1$-local weak equivalences. By inverting the $\A^1$-local weak equivalences, F. Morel and V. Voevodsky in \cite{mv99} constructed the $\A^1$-homotopy category $\mathcal{H}(k)$.

There exists an $\A^1$-fibrant replacement functor 
$$L_{\A^1}: \Delta^{op}(\Shv(\Sm/k)) \to  \Delta^{op}(\Shv(\Sm/k))$$
such that for any object $\mathcal{X},$ there is a canonical morphism $\mathcal{X}\to L_{\A^1}(\mathcal{X})$ and $L_{\A^1}(\mathcal{X})$  is an $ \A^1$-fibrant object.
We refer to \cite[\S 2.1, Theorem 1.66 and p. 69]{mv99} for the construction of the $\A^1$-fibrant replacement functor and more details.
\begin{definition}

For any $\mathcal{X} \in \Delta^{op}(\Shv(\Sm/k))$, $\pi_0^s(\mathcal{X})$ of $\mathcal{X}$ on $\Sm/k$ 
is the Nisnevich sheafification of the presheaf 
\begin{center}
$U\mapsto \pi_0(\mathcal{X})(U)= Hom_{\mathcal{H}_s(k)}(U, \mathcal{X})$
\end{center}
for $ U\in \Sm/k$.
\end{definition}
\begin{definition}

For any $\mathcal{X} \in \Delta^{op}(\Shv(\Sm/k))$, the sheaf of $\A^1$-connected components $\pi_0^{\A^1}(\mathcal{X})$ of $\mathcal{X}$ on $\Sm/k$ 
is defined as $\pi_0^s(L_{\A^1}(\mathcal{X}))$.
\end{definition}

\begin{definition}
\label{invariant}
A presheaf (or sheaf) of sets $\mathcal{F}$ on $\Sm/k$ is defined to be $\A^1$-invariant if the morphism $$ \mathcal{F}(U)\to\mathcal{F}(U\times\A^1)$$ induced by the canonical projection $U\times\A^1 \to U$ is a bijection for all $U\in \Sm/k.$
\end{definition}
\begin{remark}\label{a1rigid}
A scheme $X$ over $\Spec k$ is called $\A^1$-rigid, if the associated presheaf of sets $h_X:=\Hom_{Sch/k}(-, X)$ is $\A^1$-invariant in the sense of above Definition~\ref{invariant}.  
\end{remark}
The presheaf $U\mapsto \pi_0(L_{\A^1}(\mathcal{X}))(U)$ is $\A^1$-invariant, but that its Nisnevich sheafification $\pi_0^{\A^1}(\mathcal{X})$ is $\A^1$-invariant is conjectured by  F. Morel [Conjecture~\ref{morelconj}].

There is an analogous notion to path-connected components of a topological space in this setting. We recall the definition of the  \textit{Morel-Voevodsky singular functor} $Sing_*^{\A^1}$. 
Let $\Delta_{\bullet}$ be the cosimplicial sheaf associated  to the cosimplicial scheme 
$$\Delta_{n}=\Spec \dfrac{k[x_0, x_1,\cdots,  x_n]}{(\sum x_i-1)}.$$
For a simplcial presheaf $\mathcal{X}$ (or sheaf $\mathcal{X}$), $Sing_*^{\A^1}(\mathcal{X})$ is defined as the diagonal of the bisimplicial presheaf (or sheaf) $\underline{\Hom}(\Delta_{\bullet}, \mathcal{X})$, where $\underline{\Hom}$ is the internal hom.

\begin{definition}Let $\mathcal{F}$  be a sheaf of sets and $U$ be an essentially  smooth scheme over a field $k$. Let $n\geq 0$ be an integer. Let $\sigma_0: U\simeq U\times \{0\}\to U\times\A^1$ and $\sigma_1: U\simeq U\times \{1\}\to U\times\A^1$ be the closed immersions. 
\begin{enumerate}\label{a1chain}
\item $t_1, t_2 \in \mathcal{F}(U)$ are $\A^1$-homotopic if there is $h\in \mathcal{F}(U\times\A^1) $, such that $\sigma_0^*(h)=t_1$ and $\sigma_1^*(h)=t_2$. 
\item $t_1, t_2 \in \mathcal{F}(U)$ are $\A^1$-chain homotopic, if there is a sequence of $(h_1,\cdots, h_r)$ of $\A^1$-homotopies $ h_j\in \mathcal{F}(U\times\A^1) $, such that $\sigma_0^*(h_1)=t_1,$ 
 $\sigma_0^*(h_{i+1})= \sigma_1^*(h_i)$ and $ \sigma_1^*(h_r)=t_2$. 
\end{enumerate}
\end{definition}

\begin{definition} 
\label{aonechaindef}
Let $\mathcal{F}$ be a Nisnevich sheaf of sets on $\Sm/k$. Define the sheaf of  $\A^1$-chain connected components of $\mathcal{F}$ denoted by $\mathcal{S}(\mathcal{F})$ to be the Nisnevich sheafification of the presheaf on $\Sm/k$
\begin{center}
$V\mapsto \pi_0(Sing_*^{\A^1}\mathcal{F}(V))$
\end{center}
for $V\in \Sm/k$. In other words, $\mathcal{S}(\mathcal{F})$ agrees with the Nisnevich sheafification of the presheaf on $\Sm/k$
\begin{center}
$V\mapsto \mathcal{F}(V)/\sim$
\end{center}
for $V\in \Sm/k$, where $\sim$ is the equivalence relation given by $\A^1$-chain homotopy.
\end{definition}

%
%
 
 From here onwards, let $X \in Sch/k$, the category of schemes over $k$. We denote the representable sheaf associated to the scheme $X$ also by $X$. Let $\mathcal{S}(X)$ denote the sheaf of $\A^1$-chain connected components of $X$. 
 
 We denote $\mathcal{S}^0(X):=X$ and $\mathcal{S}^1(X):=\mathcal{S}(X)$.
  One can iterate this notion and define for $n>1$, $\mathcal{S}^n(X):=\mathcal{S}(\mathcal{S}^{n-1}(X))$ by applying the Definition~\ref{aonechaindef} to the sheaf $\mathcal{S}^{n-1}(X)$.

 We have a sequence of morphisms of sheaves (all arrows are epimorphisms)
 \begin{center}
 $X\to \mathcal{S}(X)\to \mathcal{S}^2(X)\to\cdots\to  \mathcal{S}^n(X)\to  \mathcal{S}^{n+1}(X) \to\cdots       $
 \end{center}
The \textit{universal $\A^1$-invariant quotient of $X$} (denoted by $\mathcal{L}(X)$) is defined as
 \begin{equation}
 \label{uniquotient}
     \mathcal{L}(X):=\colim_{n\geq 0}  \mathcal{S}^n(X)
 \end{equation} the colimit of the sequence of the sheaves $\{\mathcal{S}^n(X)\}_{n\geq 0}$. 
There is a canonical map  $\pi_0^{\A^1}(\mathcal{X})\to \mathcal{L}(\mathcal{X}) $ which is an isomorphism if and only if the sheaf $\pi_0^{\A^1}(\mathcal{X})$ is $\A^1$-invariant
\cite[Theorem 2.13, Remark 2.14, Corollary 2.18]{bhs}. Thus, to conclude $\A^1$-invariance of $\pi_0^{\A^1}(\mathcal{X})$ for a scheme $X$, one strategy is to study the iterations $\mathcal{S}^n(X)$ as $n$ varies. For this one needs to study the $\A^1$-homotopies of the sheaves $\mathcal{S}^n(X)$. An $\A^1$-homotopy of the sheaves $\mathcal{S}^n(X)$ can be ``lifted" Nisnevich locally to \textit{ghost homotopies} on $X$.
In the rest of this section, we recall the notions of \textit{ghost homotopies} and the \textit{space of a ghost homotopy} from \cite{bhs}. We also recall results on ghost homotopies of blow up from \cite{bs19} relevant to our setting. 
\begin{definition}\label{ghostdef}

Let $\mathcal{F}$  be a sheaf of sets and $U$ be an essentially  smooth scheme over a field $k$. Let $n\geq 0$ be an integer. Let $t_1, t_2 \in \mathcal{F}(U)$. We define the notion of $n$-ghost homotopy from $t_1$ to $t_2$, inductively on $n$.
\begin{enumerate}
    \item  $t_1, t_2 \in \mathcal{F}(U)$ are 0-ghost homotopic if $t_1, t_2$ are $\A^1$-chain homotopic in the sense of Definition~\ref{a1chain}. 
 \item  Suppose the notion of $m$-ghost homotopy is defined for $0\leq m<n$. Given $t_1, t_2 \in \mathcal{F}(U)$, an $n$-ghost homotopy from $t_1$ to $t_2$ consists of following data
 $$  \mathcal{H}=(V\to\A^1_U, W\to V\times_{\A^1_U}V, \widetilde{\sigma}_0, \widetilde{\sigma}_1, h, \mathcal{H}^W)   $$
where 
\begin{enumerate}
    \item $V\to\A^1_U$ is a Nisnevich cover of $\A^1_U$.
    \item For $i=0,1$, $\widetilde{\sigma}_i$ is a lift $U\to V$ of ${\sigma}:U\to U\times \{i\} \to\A^1_U $.
    \item $ W\to V\times_{\A^1_U}V$ is a Nisnevich cover of $ V\times_{\A^1_U}V$.
    \item $h:V\to \mathcal{F}$ is  a morphism such that $h\circ \widetilde{\sigma}_i=t_i$ for $i=1, 2$.
    \item $\mathcal{H}^W=(h_1, \cdots, h_r)$ is a chain of $(n-1)$-ghost homotopies connecting the two morphisms $$ W\to V\times_{\A^1_U}V\rightrightarrows V$$ where $V\times_{\A^1_U}V\xrightarrow{pr_i} V$ are the projections for $i=1,2$.
    
\end{enumerate}
\end{enumerate}
\end{definition}

We also recall the notion of the \textit{total space of a ghost homotopy}. 

\begin{definition}
Let $\mathcal{F}$  be a sheaf of sets and $U$ be an essentially  smooth scheme over a field $k$. Let $n\geq 0$ be an integer. For an $n$-ghost homotopy on $U$, inductively on $n$, we define the \textit{total space of the $n$-ghost homotopy} as a scheme denoted by $Sp(\mathcal{H})$ and morphisms $f_{\mathcal{H}}:Sp(\mathcal{H})\to U$ and $h_{\mathcal{H}}:Sp(\mathcal{H})\to\mathcal{F}$. 
 \begin{enumerate}
     \item For a 0-ghost homotopy, define $ Sp(\mathcal{H}):=\A^1\times U.$
     The morphism $f_{\mathcal{H}}:Sp(\mathcal{H})\to U$ is the canonical projection $\A^1\times U\to U$ and 
     $h_{\mathcal{H}}:Sp(\mathcal{H})\to\mathcal{F}$ is given by the homotopy $h$.
     \item Suppose the space $ Sp(\mathcal{H})$ is defined for $(n-1)$-ghost homotopy. Suppose an $n$-ghost homotopy is given by 
 $$  \mathcal{H}=(V\to\A^1_U, W\to V\times_{\A^1_U}V, \widetilde{\sigma}_0, \widetilde{\sigma}_1, h, \mathcal{H}^W)  $$
 with $\mathcal{H}^W=(h_1, \cdots, h_r)$ a chain of $(n-1)$-ghost homotopies. Let $Sp(\mathcal{H}_j)$ be the total space of the $(n-1)$-ghost homotopies $h_j$, then define
 $$Sp(\mathcal{H}):=V\coprod \Big(\coprod_{j=1}^r Sp(\mathcal{H}_j)\Big).$$
 The morphisms $f_{\mathcal{H}}:Sp(\mathcal{H})\to U$ and $h_{\mathcal{H}}:Sp(\mathcal{H})\to\mathcal{F}$ are given by the morphisms on various components. On $V$, $f_{\mathcal{H}}\mid_V:=V\xrightarrow{h} \A^1_U\to U$, and on $Sp(\mathcal{H}_j)$, $f_{\mathcal{H}}\mid {Sp(\mathcal{H}_j)}$ is defined as $$Sp(\mathcal{H}_j)\xrightarrow{f_{{\mathcal{H}}_j}} W\to  V\times_{\A^1_U}V \xrightarrow{pr_i} V\to \A^1_U\to U.$$
 Here $pr_i$
  can be either of the two projections, as the composition
 does not depend on the choice.
 The morphism $h_{\mathcal{H}}$ is defined as 
 $h_{\mathcal{H}}\mid_V:=V\xrightarrow{h}\mathcal{F}$ and $h_{\mathcal{H}}\mid_{Sp(\mathcal{H}_j)} :Sp({\mathcal{H}}_j)\xrightarrow{h_{{\mathcal{H}}_j}} \mathcal{F}$.
 \end{enumerate}
\end{definition}

\begin{section}{Main results}
\label{S3}
We assume throughout the rest of the paper that $k$ is a perfect field. 
We recall the definition of $\pone$-fibration. 
\begin{definition}
A morphism of $k$-schemes $\pi: X\to B$ over $\Spec k$ is called $\pone$-fibration if $\pi$ is a smooth, proper morphism such that for every point $b\in B$, the fiber $\pi^{-1}(b)$ is isomorphic to $\pone_{k(b)}$. 
\end{definition}

\begin{remark}\label{rmk0}
If a morphism of $k$-varieties $\pi: X\to B$ over $\Spec k$ is a $\pone$-fibration, then $\pi$ is \'etale-locally trivial $\pone$-fiber bundle (\cite[Lemma 3.12]{bhs}). 
\end{remark}

We consider the geometric situation we are interested in. The notation considered in the paragraph \# below, will be used throughout the rest of the paper, unless mentioned otherwise. 
 \paragraph{\#}
 Let $B$ be a smooth proper variety over $k$ of dimension 2.
  Let $X$ be a smooth proper variety over $k$ and $\pi: X\to B$ be a morphism of schemes over $k$ such that $\pi$ is a $\mathbb{P}^1$-fibration. Let $W$ be a smooth irreducible curve in $B$. Let $Z$ be a smooth closed subscheme of $X$ such that $Z$ is 
 a section of the $\pone$-fibration on $W$, pulled back from $X$. 
 Let $\widetilde{X}$ be the blow up of $X$ along the closed subscheme $Z$. 
 We summarize the above paragraph in the following commutative diagram. 
  $$\xymatrix{ 
                \~X:=Bl_Z X \ar[d] & &\\
       X\ar[d]^{\pi(\pone-\text{fib.})} & \pi^{-1}(W)\ar@{_{(}->}[l]\ar[d] & Z\ar@{_{(}->}[l]\\
             B  &W\ar@{_{(}->}[l]\ar[ru]^{s} &
 }$$
 where $s$ is a section of the $\pone$-fibration $\pi\mid_{\pi^{-1}(W)}:  \pi^{-1}(W)\to W$ so that $s(W)=Z$.

 In order to understand $\mathcal{S}^n(\~X)$ we need to understand its Nisnevich stalks \ie $\mathcal{S}^n(\~X)(U)$ for $U=\Spec R$, where $(R, \mathfrak{m})$ is the Henselization of the local ring at a smooth point of a variety over $k$. Given two sections $\alpha_1, \alpha_2\in\mathcal{S}^n(\~X) $, we need to study the $n$-ghost homotopies $h_{\mathcal{H}}:Sp({\mathcal{H}})\to \widetilde{X}$ of $\~X$ from $\alpha_1$ to $\alpha_2$.

 
 One important situation is when there is a morphism $\gamma:U\to B$ such that the $n$-ghost homotopy $h_{\calH}$ lies over a morphism $\gamma$ (in the sense of the Definition~\ref{liesovergamma}), so that we have the following commutative diagram 
 $$\xymatrix{ 
    Sp(\calH) \ar[dd]^{f_{\calH}} \ar[r]^{h_{\calH}}               &  \~X \ar[d]\\
        & X\ar[d]^{\pi}\\
            U\ar[r]^{\gamma}    & B 
 }$$
In the case of, an $n$-ghost homotopy satisfies the above property, we prove the following proposition. 

\begin{proposition}
\label{n-ghostlyingovergamma}
Let $\alpha_1$, $\alpha_2$ be sections of $\widetilde{X}$ over $U$ which are connected by an $n$-ghost homotopy $h_{\mathcal{H}}:Sp({\mathcal{H}})\to \widetilde{X}$ for some $n>0$. We further assume that, there is a morphism $\gamma:U\to B$ such that the $n$-ghost homotopy lies over the morphism $\gamma$ in the sense of the Definition~\ref{liesovergamma}. Then the sections  $\alpha_1$, $\alpha_2$ are 1-ghost homotopic and map to the same element in  $\pi_0^{\A^1}(\widetilde{X})(U)$.

\end{proposition}
We prove this proposition in Section~\ref{S4}. This proposition will serve as key step in the proofs of the main results of the paper. In subsection~\ref{a1rigidsurface}, we consider the situation when the base $B$ is an $\A^1$-rigid surface and in subsection~\ref{nonunisurface}, we consider when the base $B$ is a non-uniruled surface.
\subsection{Threefolds admitting $\pone$-fibration over $\A^1$-rigid surface}

Keeping the notation from (\#), we will prove the following
\begin{theorem}
\label{a1rigidsurface}
  Let $B$, $X$ and $\widetilde{X}$ be as described in (\#). Moreover assume that $B$ is $\A^1$-rigid. 
 Then 

\begin{enumerate}
\item $\mathcal{S}(X)\simeq B$
\item $\mathcal{S}^2(\widetilde{X})\simeq\mathcal{S}^3(\widetilde{X})$
\item $\mathcal{S}(\widetilde{X})\to \mathcal{S}^2(\widetilde{X})$ is not a monomorphism.

\end{enumerate}

\end{theorem}

\begin{proof}[Proof of (1) :]
We need to show that the canonical epimorphism $\mathcal{S}(X)\to B$ is a monomorphism of Nisnevich sheaves. 
Hence it is enough to show this for every stalk in the Nisnevich topology. The stalks in the Nisnevich topology are given by  $U$, the spectrum of an essentially smooth Henselian local $k$-algebra. Now we show that the canonical map $$\mathcal{S}(X)(U)\to B(U)$$ in injective.
We have the commutative diagram 
$$\xymatrix{
X(U)\ar@{->>}[rd] \ar@{->>}[d]& \\
\mathcal{S}(X)(U)\ar[r] & B(U)
 }
$$
 Let $\alpha_1$, $\alpha_2\in X(U)$ be such that these map to the same element in $B(U)$. 
 Then by Remark~\ref{rmk0}, the sections $\alpha_1, \alpha_2:U\to X$ factor through the pullback $X\times_ {\gamma, B} U\simeq\pone_U\to X$. Hence, the induced sections $\alpha_1, \alpha_2: U\to\pone_U$ are $\A^1$-chain homotopic \ie $\alpha_1=\alpha_2$ in $\mathcal{S}(X)(U)$. So we have proved that the epimorphism $\mathcal{S}(X)\to B$ is a monomorphism, hence an isomorphism as claimed. 
\end{proof}

\begin{proof}[Proof of (2) :] We need to show that the canonical epimorphism $\mathcal{S}^2(X)\to \mathcal{S}^3(X)$ is a monomorphism of Nisnevich sheaves. 
Hence it is enough to show for $U$, the spectrum of an essentially smooth Henselian local $k$-algebra, that the canonical map $$\mathcal{S}^2(\~X)(U)\to \mathcal{S}^3(\~X)(U)$$ is injective.

Let $\alpha_1, \alpha_2\in \widetilde{X}(U)$, which are connected by a $2$-ghost homotopy $h_{\mathcal{H}}:Sp({\mathcal{H}})\to \widetilde{X}$. 
Since $B$ is $\A^1$-invariant, for the $2$-ghost homotopy $h_{\mathcal{H}}:Sp({\mathcal{H}})\to \~X\to X$ over $U$, the homotopy $\pi\circ h_{\mathcal{H}}: Sp({\mathcal{H}})\to X\to  B$ is constant (see\cite[Lemma 2.12]{bs19}). Thus, if two morphisms $\alpha_1, \alpha_2: U\to \~X$ are $2$-ghost homotopic, then the compositions
 $\pi\circ\alpha_1$ and $\pi\circ\alpha_2$ are equal, say $\gamma: U\to B$.  Thus, the $2$-ghost homotopy connecting the sections $\alpha_1, \alpha_2\in \widetilde{X}(U)$ lies over the morphism $ \gamma: U\to B$. Now by applying Proposition~\ref{n-ghostlyingovergamma}, we get that the sections  $\alpha_1, \alpha_2$ are 1-ghost homotopic \ie $\alpha_1=\alpha_2$ in $\mathcal{S}^2(\~X)(U)$.

 This proves that $\mathcal{S}^2(\widetilde{X})(U)\simeq \mathcal{S}^3(\widetilde{X})(U)$, for all Henselian local $k$-algebras. Hence, the morphism $\mathcal{S}^2(\widetilde{X}) \to \mathcal{S}^3(\widetilde{X})$ of Nisnevich sheaves is an isomorphism. 
\end{proof}

To see that $\mathcal{S}(\widetilde{X})\neq \mathcal{S}^2(\widetilde{X})$, we use a criterion for two sections of $\widetilde{X}$ to be $\A^1$-chain homotopic.
Keeping the notations from the \textit{proof of the claim 2} above, we recall the following proposition.
\begin{proposition} (\cite[Proposition 4.10]{bs20})
\label{aonechain}
Let $r_0\in \mathfrak{m}-\{0\}.$ Let $r_1, r_2\in \mathfrak{m}-\{0\} $ be such that $r_1\mid r_0$,  $r_2\mid r_0$ but $r_0\nmid r_1$ and $r_0\nmid r_2$. Let $r'=r_0/r_1.$ 
There exists an $\A^1$-chain homotopy connecting $\alpha_{1}$ to $\alpha_{2}$ which lifts to $X_{\gamma}$ if and only if
\begin{enumerate}

\item $r_2$ is a unit multiple of $r_1$.
\item $r_2/r_1-1\in rad(<r_1>) + rad(<r'>)\subseteq rad(<r_1, r'>).$
\end{enumerate}
\end{proposition}

%
%
%
%

\begin{proof}[Proof of claim 3:]
We show that $\mathcal{S}(\widetilde{X})(U)\neq \mathcal{S}^2(\widetilde{X})(U)$ for some Henselian local $k$-algebra $U$. \\
Let $U= \Spec R$, where $R=k[x, y]^h_{(x, y)}$, where $k[x, y]$ is the ring of polynomials in the indeterminates $x, y$ and $(x, y)$ is the maximal ideal generated by $x$ and $y$.
Let $r_0= x(y^2+x), r_1=x$ and $r_2=x(1+y)$. \\
We observe that
$$ rad(<r_1>) + rad(<r'>)= <x, y^2>$$ and $$r_2/r_1-1=y\notin rad(<r_1>) + rad(<r'>).$$
It follows from the Proposition \ref{aonechain} above that, the morphisms $\alpha_{1}, \alpha_{2}$ corresponding to $r_1, r_2$ respectively, are not $\A^1$-chain homotopic. Furthermore, the morphisms $\alpha_{1}, \alpha_{2}$ are 1-ghost homotopic, since $r_2/r_1-1=y\in rad(<x, y^2>)$ as seen earlier in the \textit{proof of the claim 2}.
\end{proof}

\begin{corollary}
\label{cor1}

Let $\widetilde{X}$ be as in the Theorem~\ref{a1rigidsurface}.
 Then the sheaf $\pi_0^{\A^1}(\widetilde{X})$ is $\A^1$-invariant.

\end{corollary}
\begin{proof}
It follows from Theorem~\ref{a1rigidsurface} and \cite[Lemma 4.1]{bhs} that the canonical morphism 
$$\pi_0^{\A^1}(\widetilde{X})\to \mathcal{S}^2(\widetilde{X})\simeq   \mathcal{L}(\widetilde{X}) $$
admits a retract. Then the corollary follows from~\cite[Lemma 2.16]{bhs}.
\end{proof}

 \subsection{Threefolds admitting $\pone$-fibration over non-uniruled surface}
%

Keeping the notation from Section~\ref{S3} (\#), we will prove the following

\begin{theorem}
\label{nonunisurface}

 Let $X$ and $\widetilde{X}$ be as in (\#). Further, assume that $B$ is non-uniruled. 
 Then 
\begin{enumerate}
\item $\mathcal{S}(X)\simeq \mathcal{S}^2(X)$.
\item $\mathcal{S}^2(\widetilde{X})\simeq\mathcal{S}^3(\widetilde{X})$.
\end{enumerate}

\end{theorem}
\begin{proof}[Proof of (1):]
 The proof of this follows along the lines of the proof of ~\cite[Theorem 3.14]{bhs}.
\end{proof}

\begin{proof}[Proof of (2):]
Since $B$ is not uniruled, we recall from ~\cite[Chapter VI, Prop 1.3]{kollar} that, for every $k$-variety $T$ and a rational map $\pone\times T\dashrightarrow  B$ either 

\begin{enumerate}
 
\item the rational map $\pone\times T\dashrightarrow B$ is not dominant or 
\item for every $t\in T$, the induced map $\pone_{k(t)}\dashrightarrow B$ is constant.
\end{enumerate}
We can take $T$ above to be an essentially smooth $k$-scheme.

Let $U=\Spec R $ be a smooth Henselian local $k$-algebra as in the Notation~\ref{notation}(1).
We have the commutative diagram 
$$\xymatrix{
\~X(U)\ar@{->>}[rd] \ar@{->>}[d]& \\
\mathcal{S}^2(\~X)(U)\ar[r] & \mathcal{S}^3(\~X)(U)
 }
$$
  Let $\alpha_1$, $\alpha_2\in \~X(U)$ be sections of $\widetilde{X}$ over $U$, such that both map to the same element in $\mathcal{S}^3(\~X)(U)$. In other words, $\alpha_1, \alpha_2$ are connected by a $2$-ghost homotopy $h_{\mathcal{H}}:Sp({\mathcal{H}})\to \widetilde{X}$.
   Suppose the $2$-ghost homotopy is given by 
 $$  \mathcal{H}=(V\to\A^1_U, W\to V\times_{\A^1_U}V, \widetilde{\sigma}_0, \widetilde{\sigma}_1, h, \mathcal{H}^{W})$$
 with $\mathcal{H}^{W}=(\mathcal{H}^{1}_{1}, \cdots, \mathcal{H}^{1}_{r})$ a chain of $ 1$-ghost homotopies as in the Definition~\ref{ghostdef}. 
Suppose the $1$-ghost homotopy $\mathcal{H}^1_m$ is given by 
 $$  \mathcal{H}^1_m=(V_m^1\to\A^1_{W}, W_m^1\to V_m^1\times_{\A^1_{W^2}}V_m^1, \widetilde{\sigma}_0, \widetilde{\sigma}_1, h, \mathcal{H}^{W^1_m})$$ where
 $\mathcal{H}^{W^1_m}=(\calH^0_1, \calH^0_2, \dots, \calH^0_s)$ is a chain of 0-ghost homotopies, in other words, $\calH^0_i$ are $\A^1$-chain homotopies on $W^1_m$.   
 
 The space of the 2-ghost homotopy $\mathcal{H}$ is given by, $$Sp(\mathcal{H})=V\coprod \Bigg(\coprod_{m=1}^{r} V_m^{1}\Bigg)
 \coprod \Bigg(\coprod_{m=1}^{r} \A^1_{W_m^1}\Bigg).$$
We write $h\mid _{ \A^1_{W_m^1}}$ or $h \mid _{V_m^i}$ for the respective restrictions of the morphism $h_{\mathcal{H}}:Sp(\mathcal{H})\to \widetilde{X}$. 

We summarize the data in the following diagrams
$$\xymatrix{
                &                                                  &                & U \ar@<-.5ex>[d]^{~\sigma_1} \ar@<.5ex>[d]_{~\sigma_0} \ar@<-.5ex>[ld]^{~\~{\sigma_1}} \ar@<.5ex>[ld]_{~\~{\sigma_0}}\\
W \ar[r]\ar@<-.5ex>[rrd] \ar@<.5ex>[rrd] & V\times_{\A^1_U}V \ar@<-.5ex>[r]_{~pr_2} \ar@<.5ex>[r]^{~pr_1}& V \ar[r] \ar[d]^{h}& \A^1_U \\
                &                                                  &          \~X       & 
}$$
We have a chain of 1-ghost homotopies  $\mathcal{H}^{W}=(\mathcal{H}^{1}_{1}, \cdots, \mathcal{H}^{1}_{r})$ between the two arrows from $W\to \~X$ \ie for $1\leq m\leq r$,

$$\xymatrix{
                &                                                  &                & W \ar@<-.5ex>[d]^{~\sigma_1} \ar@<.5ex>[d]_{~\sigma_0} \ar@<-.5ex>[ld]^{~\~{\sigma_1}} \ar@<.5ex>[ld]_{~\~{\sigma_0}}\\
W^1_m \ar[r]\ar@<-.5ex>[rrd] \ar@<.5ex>[rrd] & V^1_m\times_{\A^1_U}V^1_m \ar@<-.5ex>[r]_{~pr_2} \ar@<.5ex>[r]^{~pr_1}& V^1_m \ar[r] \ar[d]^{h^1_m}& \A^1_{W} \\
                &                                                  &          \~X       & 
}$$
where we have a chain of $\A^1$-homotopies between the two arrows from $W^1_m\to \~X$.

We write $V=\coprod_{i\in I} V_{i}$, where $V_i$ are irreducible for each $i\in I.$ Write $$W=\coprod_{i, j\in I}\coprod_{k\in K_{ij}} W^{ij}_k$$
where each $W^{ij}_k$ is irreducible and $K_{ij}$ is the indexing set such that $ \coprod_{k\in K_{ij}} W^{ij}_k\to V_i\x_{\A^1_U}V_j$ is a Nisnevich cover. Similarly, we write $V^1_m=\coprod_{i\in I_m} V^1_{mi}$, where $V^1_{mi}$ are irreducible for each $i\in I_m$. Again write $$W^1_m=\coprod_{i, j\in I_m}\coprod_{k\in K_{mij}} W^{ij}_{mk}$$
where each $W^{ij}_{mk}$ is irreducible and $K_{mij}$ is the indexing set such that $ \coprod_{k\in K_{mij}} W^{ij}_{mk}\to V^1_{mi}\x_{\A^1_U}V^1_{mj}$ is a Nisnevich cover.
We consider the maps $h^1_m: V^1_m\to \~X\to X\to B$ and $h: V\to \~X\to X\to B$ and denote the restrictions of $h^1_m$, to the irreducible components $V_{mi}$ as $h^1_m|_{V^1_{mi}}$, the restrictions of $h$, to the irreducible components $V_{i}$ as $h|_{V_{i}}$, etc. 

Assume that for every $m, i, j, k$, $(h^1_m|_{V^1_{mi}})|_{W^{ij}_{mk}}=(h^1_m|_{V^1_{mj}})|_{W^{ij}_{mk}}$. Then  $(h^1_m|_{V^1_{mi}})|_{V^1_{mi}\x_{\A^1_U}V^1_{mj}}=(h^1_m|_{V^1_{mj}})|_{V^1_{mi}\x_{\A^1_U}V^1_{mj}}$. Hence the morphisms $(h^1_m|_{V^1_{mi}})$ glue to give a morphism $h:\A^1_W\to B$. Now further if $h:\A^1_W\to B$ is a constant homotopy,
then in particular, $h(0)=h(1)$, i.e for every $i, j, k$, $(h|_{V_{i}})|_{W^{ij}_{k}}=(h|_{V_{j}})|_{W^{ij}_{k}}$. Hence $(h|_{V_{i}})|_{V_i\x_{\A^1_U}V_j}=(h|_{V_{j}})|_{V_i\x_{\A^1_U}V_j}$. Thus, $h|_{V_{i}}$ glue to give a morphism $h: \A^1_U\to B$. Again if the homotopy $h: \A^1_U\to B$ is constant, then in particular $\alpha_1=h(0)=h(1)=\alpha_2:U\to B$ and the entire 2-ghost homotopy lies over a map $\gamma: U\to B$.  Now by the Proposition~\ref{n-ghostlyingovergamma}, the sections $\alpha_1$ and $\alpha_2$ are 1-ghost homotopic.

On the other hand, if either of the homotopies appearing above are non-constant, then the 2-ghost homotopy factors through a 1-dimensional closed subscheme $C\subseteq B$. 
Thus, the given 2-ghost homotopy $h_{\mathcal{H}}:Sp({\mathcal{H}})\to \widetilde{X}$ when composed with 
$\widetilde{X}\to X$ factors through the $\pone$-fibration over $C$. Write $C=\cup_{j\in J} C_j$, where each $C_j$ is an irreducible component of $C$. Since the morphisms appearing in the 2-ghost homotopy when restricted to irreducible components are dominant onto irreducible components of $C$, the 2-ghost homotopy   $h_{\mathcal{H}}:Sp({\mathcal{H}})\to \widetilde{X}\to X\to B$ factors through  the normalization $\overline{C}$ of $C$, where  $\overline{C}=\coprod_{j\in J}\overline{C_j} $ is the disjoint union of the normalization $\overline{C_j}$ of $C_j$. Thus, the 2-ghost homotopy $h_{\mathcal{H}}:Sp({\mathcal{H}})\to \widetilde{X}$ factors through the blow up of the $\pone$-fibration  $X\times_{B}\overline{C}\to \overline{C}$ along the inverse image $Z'$ of $Z$ under the morphism $X\times_{B}\overline{C} \to X$. Now, either the smooth curve $\overline{C_j}$ is isomorphic to $\pone$ or is of positive genus. In the case that  $\overline{C_j}\simeq \pone$,  the blow up is a rational surface. Hence, in this case, the result follows since $\mathcal{S}^2(Bl_{Z'}(X\times_{B}\overline{C_j}))=*$ as shown in \cite[Corollary 3.3]{sawant17}. In the other case, that $\overline{C_j}$ is a smooth projective curve of positive genus, we observe that the blow up of the $\pone$-fibration  $X\times_{B}\overline{C_j}\to \overline{C_j}$ along $Z'$ 
 is a birationally ruled surface over $\overline{C}$. Here $\overline{C_j}$ is $\A^1$-rigid, thus, given an $n$-ghost homotopy to $X\times_{B}\overline{C_j}$, after composing with $X\times_{B}\overline{C_j}\to \overline{C_j}$ lies over a morphism $\gamma:U\to \overline{C_j}$. By the similar argument, as employed in the Proposition~\ref{n-ghostlyingovergamma}, we deduce that the sections $\alpha_1$, $\alpha_2$ are 1-ghost homotopic and map to the same element in  $\pi_0^{\A^1}(\widetilde{X})(U)$.
\end{proof}

\begin{corollary} \label{cor2}
For $\widetilde{X}$ as in the Theorem~\ref{nonunisurface},
 $\pi_0^{\A^1}(\widetilde{X})$ is $\A^1$-invariant.
\end{corollary}
\begin{proof}
It follows from Theorem~\ref{nonunisurface} and \cite[Lemma 4.1]{bhs} that the canonical morphism 
$$\pi_0^{\A^1}(\widetilde{X})\to \mathcal{S}^2(\widetilde{X})\simeq \mathcal{L}(\~X)  $$
admits a retract. Then the corollary follows from~\cite[Lemma 2.16]{bhs}.
\end{proof}

\begin{remark}\label{notuninotmono}
In general for a non-uniruled surface $B$, the morphism of sheaves $\mathcal{S}(\widetilde{X})\to \mathcal{S}^2(\widetilde{X})$ might not be a monomorphism. For example, for $B$ an abelian surface, hence $\A^1$-rigid, we have $\mathcal{S}(\widetilde{X})\to \mathcal{S}^2(\widetilde{X})$ is not a monomorphism as it follows from Theorem~\ref{a1rigidsurface}(3).
\end{remark}


\begin{remark}\label{sing*}
Theorem~\ref{a1rigidsurface}(3), and the Remark \ref{notuninotmono} provide examples of $\widetilde{X}$ such that $Sing_*^{\A^1}(\widetilde{X})$ is not $\A^1$-local.
If  $\widetilde{X}$ is such that $\mathcal{S}(\widetilde{X})\to \mathcal{S}^2(\widetilde{X})$ is not a monomorphism, then $Sing_*^{\A^1}(\widetilde{X})$ is not $\A^1$-local in the sense of \cite[Definition 2.1, page 106]{mv99}. For if  $Sing_*^{\A^1}(\widetilde{X})$ were $\A^1$-local, then the morphism 
$\mathcal{S}(\widetilde{X})\to \pi_0^{\A^1}(\widetilde{X})$ would be an isomorphism. \end{remark}

\end{section}

\begin{section}{Ghost homotopies on blow up of threefolds admitting $\pone$-fibration over a surface} 
\label{S4}

For the rest of the paper, we follow the following notations.

\begin{notation}
\label{notation}

\begin{enumerate}
\item Let $U=\Spec R$, where $(R, \mathfrak{m})$ is the Henselization of the local ring at a smooth point of a variety over $k$. 
\item For a local $k$-algebra $A$, $A^h$ denotes the \textit{Henselization} of the local ring $A$.
\item $<a, b>$ denotes the ideal in $R$ generated by elements $a, b\in R$.
\item For an ideal $J$ in a ring $R$, $rad (J)$ denotes the radical of the ideal $J$.
\item For an ideal $I$ in a ring $R$ and $U=\Spec R$, $U(I)$ denotes the closed subscheme $\Spec R/I$ of $U$. 
\end{enumerate}
\label{notn}
\end{notation}

\begin{definition}\label{liesovergamma}
Let $\phi:\mathcal{F}\to \mathcal{G}$  be a morphism of Nisnevich sheaves of sets and $U$ be an essentially smooth scheme over a field $k$. Let $n\geq 0$ be an integer. We say that an $n$-ghost homotopy $\mathcal{H} $ on $\mathcal{F}$ over $U$, lies over a morphism $\gamma:U\to \mathcal{G}$, when $\gamma\circ f_{\mathcal{H}}=\phi\circ h_{\mathcal{H}}$.
\end{definition}

\begin{remark}\label{rmk1}
We note here a useful remark. Let $Y=\Spec A$ be a $k$-scheme. Let $\alpha_1, \alpha_2: U\to \pone_k\x Y$ over $Y$ \ie the following diagram of $k$-schemes commutes. 
$$\xymatrix{
U  \ar@/^/[r]^{\alpha_1}\ar@/_/[r]_{\alpha_2}\ar[rd]& \pone_k\x Y\ar[d]^{pr_Y}\\
    & Y
}$$ 
Then  $\alpha_1, \alpha_2$ factor through $\Spec A[T]$, so $\alpha_i$ are determined by the ring homomorphisms $A[T]\to R$ sending $T\mapsto r_i$. Now the linear homotopy $T\mapsto r_1(1-T)+r_2T$ determines an $\A^1$-homotopy from $\alpha_1$ to $\alpha_2$.
\end{remark}

We assume throughout the rest of the paper that $k$ is a perfect field. 
 We consider the geometric situation we are interested in. The notation considered in the paragraph \# below, will be used throughout this Section~\ref{S3} unless mentioned otherwise. 
 \paragraph{\#}
 Let $B$ be a smooth proper variety over $k$ of dimension 2.
  Let $X$ be a smooth proper variety over $k$ and $\pi: X\to B$ be a morphism of schemes over $k$ such that $\pi$ is a $\mathbb{P}^1$-fibration. Let $W$ be a smooth irreducible curve in $B$. Let $Z$ be a smooth closed subscheme of $X$ such that $Z$ is 
 a section of the $\pone$-fibration on $W$, pulled back from $X$. 
 Let $\widetilde{X}$ be the blow up of $X$ along the closed subscheme $Z$. 
 
In order to understand $\mathcal{S}^n(\~X)$ we need to understand its Nisnevich stalks \ie $\mathcal{S}^n(\~X)(U)$ for $U=\Spec R$, where $(R, \mathfrak{m})$ is the Henselization of the local ring at a smooth point of a variety over $k$. Given two sections $\alpha_1, \alpha_2\in\mathcal{S}^n(\~X) $, we need to study the $n$-ghost homotopies $h_{\mathcal{H}}:Sp({\mathcal{H}})\to \widetilde{X}$ of $\~X$ from $\alpha_1$ to $\alpha_2$.

 We further assume that, there is a morphism $\gamma:U\to B$ such that the $n$-ghost homotopy lies over the morphism $\gamma$ (in the sense of the Definition~\ref{liesovergamma}), so that we have the following commutative diagram 
 $$\xymatrix{ 
    Sp(\calH) \ar[rdd]_{f_{\calH}} \ar[rr]^{h_{\calH}} \ar@{-->}[rd]    &          &  \~X \ar[d]\\
      & X\times _{\gamma, B}U\ar[r]\ar[d] & X\ar[d]^{\pi}\\
      &      U\ar[r]^{\gamma}    & B 
 }$$
 
 The assumption on the $n$-ghost homotopy implies that, the morphism $h_{\mathcal{H}}:Sp({\mathcal{H}})\to\widetilde{X}\to X$ factors though  $X\times _{\gamma, B}U\to U$ which lifts to the blow up of $X\times _{\gamma, B}U$ along the ideal sheaf $\mathcal{I}_{\gamma}$ associated to the closed subscheme given by the inverse image of $Z$ under the morphism $X\times _{\gamma, B}U\to X$.

  We prove a series of lemmas (Lemma~\ref{lemma1}, Lemma~\ref{lemma2}, Lemma~\ref{lemma3}, and Lemma~\ref{lemma4}) depending on the  \textbf{codimension of the images in $B$ of the generic point $\eta$ and the closed point $u$ of $U$ under the morphism $\gamma:U\to B$}.

\begin{lemma}
\label{lemma1}
Let $\alpha_1$, $\alpha_2$ be sections of $\widetilde{X}$ over $U$ which are connected by an $n$-ghost homotopy $h_{\mathcal{H}}:Sp({\mathcal{H}})\to \widetilde{X}$ for some $n>0$. We further assume that, there is a morphism $\gamma:U\to B$ such that the $n$-ghost homotopy lies over the morphism $\gamma$ in the sense of the Definition~\ref{liesovergamma}. Assume that the morphism $\gamma:U\to B$ satisfies either of the following properties. 
\begin{enumerate}[label=\ref{lemma1}(\arabic*)]
\item\label{lemma11} $\gamma(u)$ is the generic point of $B$.
\item\label{lemma12} $\gamma(\eta)$ is a closed point $b$ of $B$.
\end{enumerate}
Then the sections  $\alpha_1$, $\alpha_2$ are 1-ghost homotopic and map to the same element in  $\pi_0^{\A^1}(\widetilde{X})(U)$.
\end{lemma}

\begin{proof}
First let us consider the case~\ref{lemma11}. By assumption on the morphism $\gamma: U\to B$, $\gamma(u)$ is the generic point of $B$. So $\gamma(u)\in \Spec A$ for some affine open subscheme $\Spec A$ of $B-W$. Hence the morphism
 $\gamma$ factors though $\Spec A$, since $U$ is a local affine scheme. We can further choose $\Spec A$ such that the $\pone$-fibration $\pi: X\to B$, when restricted to $\Spec A$ is trivial \ie isomorphic to $\pone\x \Spec A$. In particular, the sections $\alpha_1, \alpha_2: U\to\~X\to X$ factor through 
 $\pone\times \Spec A$. Hence,  $\alpha_1, \alpha_2$ are $\A^1$-homotopic by the Remark~\ref{rmk1}.   
 
On the other hand,  consider the case~\ref{lemma12}. By assumption, $\gamma(\eta)$ is a closed point $b$ of $B$. Then any morphism $U\to \widetilde{X}$ lying over $\gamma$ factors through the fiber of $\widetilde{X}\to B$ over $b$, which is a connected scheme with each irreducible component isomorphic to $\pone_{k(b)}$.
Hence, the morphisms $\alpha_1, \alpha_2 :U\to \widetilde{X}$ lying over $\gamma$ are $\A^1$-chain homotopic. 
\end{proof}

\begin{lemma}
\label{lemma2}
Let $\alpha_1$, $\alpha_2$ be sections of $\widetilde{X}$ over $U$ which are connected by an $n$-ghost homotopy $h_{\mathcal{H}}:Sp({\mathcal{H}})\to \widetilde{X}$ for some $n>0$. We further assume that, there is a morphism $\gamma:U\to B$ such that the $n$-ghost homotopy lies over the morphism $\gamma$ in the sense of the Definition~\ref{liesovergamma}. Assume that the morphism $\gamma:U\to B$ is such that 
\begin{enumerate}[label=\ref{lemma2}(\arabic*)]
\item $\gamma(u)=\gamma(\eta)=y$ of codimension 1 in $B.$
\end{enumerate}
Then the sections  $\alpha_1$, $\alpha_2$ are 1-ghost homotopic and map to the same element in  $\pi_0^{\A^1}(\widetilde{X})(U)$.
\end{lemma}
\begin{proof}
By the assumption on the morphism $\gamma$, $\gamma(u)=\gamma(\eta)=y$ is of codimension 1 in $B$. Let $Y=\overline{\{y\}}$ the closure of $y$ in $B$ with the reduced closed subscheme structure. 
 Let $\overline{Y}$ be the normalization of $Y$, which is a smooth curve. Since $\gamma:U\to Y$ is dominant, it factors through $\overline{Y}\to Y$. Thus, the $n$-ghost homotopy connecting $\alpha_1, \alpha_2$ factors through the $\pone$-fibration over $\overline{Y}$ given by the pullback $X\times_B \overline{Y}$.
There are two cases: 
\begin{enumerate}
\item[1)] $y\notin W$ i.e. $Y\neq W.$
 \item[2)] $y\in W$, i.e. $Y= W$
\end{enumerate}
 
We consider the cases separately. 

In the case 1) assume that $Y\neq W$. Since $\gamma(u)$ is the generic point of $Y$, $\gamma(u)\in Y-W$. Hence $\gamma:U\to Y$ factors through $Y-W\to Y$. So we can replace $Y$ by $Y-W$. 
Thus $\alpha_1, \alpha_2$ are such that these map the closed point to the generic point of $\overline{Y}$, so the morphisms $\alpha_1, \alpha_2$ factor through $\pone\times \Spec A$, where $\Spec A$ is an affine open subscheme of $\overline{Y}$. Hence  $\alpha_1, \alpha_2$ are $\A^1$-chain homotopic by the Remark~\ref{rmk1}.

 On the other hand, consider the case 2) $Y=W$. Here the morphisms $\alpha_1, \alpha_2:U\to \widetilde{X}$ factor through the total transform of $\pi^{-1}(W)$ in $\widetilde{X},$ which is the union of the exceptional divisor $E$ and the strict transform $V'$ of  $\pi^{-1}(W)$. 
Since $\gamma$ maps the generic point $\eta$ to the generic point of $W$, the $n$-ghost homotopy (in particular, the morphisms $\alpha_1, \alpha_2$) lying over $\gamma$ factors through the total stransform of $\pi^{-1}(W)$ which is a union of the exceptional divisor $E$ and the strict transform $V'$ of $\pi^{-1}(W)$. Since $Z$ is a smooth closed subscheme of $\pi^{-1}(W)$ of codimension 1, $V'$ is isomorphic to $\pi^{-1}(W)$, which is a geometrically ruled surface over the smooth curve $W$. Also $E$ is a geometrically ruled surface over the smooth curve $Z$. Thus, we observe that~\cite[Proposition 2.12]{bs19} implies that the morphisms $\alpha_1, \alpha_2:U\to E\cup V'$ are $\A^1$-chain connected. By composing with the inclusion morphism $E\cup V'\hookrightarrow \widetilde{X}$, we get that the morphisms $\alpha_1, \alpha_2:U\to \widetilde{X}$ are connected by an $\A^1$-chain homotopy.
\end{proof}

Now we prove technical lemmas (Lemma~\ref{lemmaA} and Lemma~\ref{lemmaB}), that will be used in the remaining part of this section. We fix some notation that will be used throughout the rest of this paper. 

Let $\pone_U:=\Proj R[x_0, x_1]$. Let $J\subsetneq R$ be a proper ideal of $R$. Let $X_{\gamma}$ denote the blow-up of $\pone_U$ along an ideal $I_{\gamma}=<J, x_0>$, where $<J, x_0>$ denotes the homogeneous ideal in $R[x_0, x_1]$ generated by $J$ and $x_0$. Let $\phi: X_{\gamma}\to \pone_U$ be the natural map. 

Let $\alpha_1$, $\alpha_2: U\to X_{\gamma}$ over $U$ which are connected by an $n$-ghost homotopy $h_{\mathcal{H}}:Sp({\mathcal{H}})\to \pone_U$ for some $n>0$. 

Fix $i=1, 2$. Now $\phi\circ\alpha_i: U\to X_{\gamma}\to \pone_U$ are given by a tuple $(r_i, s_i)\in R^2$ such that the ideal generated by $r_i$ and $s_i$ is unit ideal \ie $<r_i, s_i>=R$. Now since $R$ is a local domain, either $r_i$ or $s_i$ is a unit in $R$. Now since the maps $\alpha_i$ already lift to the blow up $X_{\gamma}$, we have by~\cite[Proposition 3.7]{bs19},  $$\alpha_1^*(I_{\gamma})=\alpha_2^*(I_{\gamma}).$$ 

Thus, we have $<J, r_1>=<J, r_2>.$
From this and the fact that $J$ is a proper ideal in $R$, we can conclude that $r_1$ is a unit in $R$ iff $r_2$ is a unit in $R$. Now, $\Proj R[x_0, x_1]$ is a union of Zariski open affine schemes $\Spec R[x_0/x_1]$ and $\Spec R[x_1/x_0]$. Since $U$ is a local scheme, $\alpha_i$ factors through either $\Spec R[x_0/x_1]$ or $\Spec R[x_1/x_0]$.  So if $r_1$ is a unit, then both $\alpha_1$, $\alpha_2$ factor through $\Spec R[x_1/x_0]$ and if $r_1$ is not a unit, hence $s_1$ is a unit, then both $\alpha_1$, $\alpha_2$ factor through $\Spec R[x_0/x_1]$. Thus, without loss of generality, we can assume that both the sections factor through $\Spec R[x_0/x_1]$ and are determined by $x_0/x_1\mapsto r_i\in R$. 

By~\cite[Proposition 3.7]{bs19}, applied to the Blow-up $X_{\gamma}$, 
 there exists $r\in R$ such that 
  the ideals $\alpha_1^*(I_{\gamma}), \alpha_2^*(I_{\gamma})$ and $h_{\mathcal{H}}^*(I_{\gamma})$ are generated by $r$. We have $$\alpha_1^*(I_{\gamma})=J+< r_1>=<r>=J+< r_2>= \alpha_2^*(I_{\gamma}).$$
There are two cases:
\begin{itemize}
\item[(A)] $r\in J$, thus $r\mid r_1, r_2$.
\item[(B)] $r=r_1$, so that $J\subseteq <r>$, but $r\notin J$ 
; also we have that $r_2$ is a unit multiple of $r.$  
\end{itemize}

We show below that, if the two sections $\alpha_{1}, \alpha_{2}$ satisfy the condition of the case (A), then the sections are $\A^1$-homotopic. But on the other hand, the necessary condition in (B) is not sufficient to conclude that the sections $\alpha_1, \alpha_2$ are $n$-ghost homotopic. In this case, we get a further necessary condition, which suffices to show that the sections $\alpha_1, \alpha_2$ are $1$-ghost homotopic.

\begin{lemma}\label{lemmaA}
Let $\alpha_1$, $\alpha_2: U\to X_{\gamma}$ be sections of $X_{\gamma}$ over $U$ which are connected by an $n$-ghost homotopy $h_{\mathcal{H}}:Sp({\mathcal{H}})\to X_{\gamma}$ for some $n>0$. Assume that $r\in J$. Then the sections $\alpha_1$, $\alpha_2$ are $\A^1$-homotopic.
\end{lemma}
\begin{proof}
Consider $h:\A^1_U=\Spec R[T]\to \Spec R[x_0/x_{1}]$ given by 
  \begin{center}

  $x_0/x_1\mapsto r_1(1-T)+r_2T$.
  
  \end{center}
Note that the homotopy $ h: U\times \A^1 \to \mathbb{P}^1_U$ lifts to $X_{\gamma}$ as the pullback
\begin{align*}
h^*(\mathcal{I}_{\gamma})&=<J, h^*(x_0)>\\
                &= <J, r_1(1-T)+r_2T>\\
                &=<r>
            \end{align*}
 i.e. the ideal $h^*(\mathcal{I}_{\gamma})$ is a locally principal ideal of $R[T].$ 
Thus, the morphism $h: U\times \A^1 \to \mathbb{P}^1_U$ lifts to $X_{\gamma}$ which gives an $\A^1$-homotopy joining $\alpha_{1}$ and $\alpha_{2}$.  
\end{proof}

\begin{lemma}\label{lemmaB}
Let $\alpha_1$, $\alpha_2: U\to X_{\gamma}$ be sections of $X_{\gamma}$ over $U$ which are connected by an $n$-ghost homotopy $h_{\mathcal{H}}:Sp({\mathcal{H}})\to X_{\gamma}$ for some $n>0$. Assume that $r=r_1$, so that $J\subseteq <r>$, but $r\notin J$; also we have that $r_2$ is a unit multiple of $r.$  Then 
\begin{equation}
\dfrac{r_2}{r_1}-1\in rad<r_1, s/r_1>,  ~~\text{for all $s\in J$}.
\end{equation}
Conversely, if $r=r_1$, so that $J\subseteq <r>$, but $r\notin J$; also we have that $r_2$ is a unit multiple of $r.$  Then there exists a 1-ghost homotopy from $\alpha_1$ to $\alpha_2$. Furthermore, $\alpha_1$ and $\alpha_2$ map to the same element in $\pi_0^{\A^1}(X_{\gamma})(U)$.
\end{lemma}
\begin{proof}
Let $h_{\mathcal{H}}:Sp({\mathcal{H}})\to \mathbb{P}^1_U$ be an $n$-ghost homotopy, for $n>0$, connecting $\alpha_1$ and $\alpha_2$ which lifts to $X_{\gamma}$. 
Then we show that, 
\begin{equation}
\dfrac{r_2}{r_1}-1\in rad<r_1, s/r_1>,  ~~\text{for all $s\in J$.}
\end{equation}
 Consider $X_r:= Bl _{\mathcal{I}_r}\mathbb{P}^1_U$  the blow up of $\mathbb{P}^1_U$ along the ideal sheaf $\mathcal{I}_r$ associated to the homogenous ideal $<r, x_0>$ in $R[x_0, x_1]$. It is clear to see that the $n$-ghost homotopy $h_{\mathcal{H}}$ lifts to $X_r$, indeed $h_{\mathcal{H}}^*(<r, x_0>)=<r, h_{\mathcal{H}}^*(x_0)>=<r>$. In particular, the sections $\alpha_1$ and $\alpha_2$ lift to  $X_r$ (denoted by $\alpha_1'$ and $\alpha_2'$ respectively). 
 
$X_r$ is given in the chart $x_1\neq 0$ by the equation: $ry_1=\dfrac{x_0}{x_1}y_2.$
 
 Fix $s\in J$. Consider the extension $\widetilde {<s, x_0>}$ of the ideal $<s, x_0>$ in $\mathbb{P}^1_U$ to $X_r$ : On the chart $y_2\neq 0$: 
 \begin{align*}
 \widetilde {<s, x_0>}&=<s, \dfrac{x_0}{x_1}>\\
                            &=<s, \dfrac{x_0}{x_1}><1, \dfrac{y_1}{y_2}>  \\
                            &=<\dfrac{s}{r}, \dfrac{y_1}{y_2}><r, \dfrac{x_0}{x_1}>. 
 \end{align*}
  
On the chart $y_1\neq0$, we have $r= \dfrac{y_2}{y_1}\dfrac{x_0}{x_1}$, hence

 \begin{align*}
 \widetilde {<s, x_0>}&=<s, \dfrac{x_0}{x_1}>\\
                            &=<r, \dfrac{x_0}{x_1}>  ~(\text{since $s\in <r>$}).
 \end{align*}
The exceptional locus in the blow up $X_r$ is given by the ideal $<r, x_0>$. The projection onto the  second component given by 
\begin{equation}\label{blXr}
 X_r   \to \mathbb{P}^1_U\times_U \Proj R[y_1, y_2]  \to \Proj R[y_1, y_2]
\end{equation}
is the blow up along the closed subscheme $\mathcal{Z}(<r, y_2>)$. \\
We define $X_{r, s/r}\to X_r$ as the blow-up of $X_r$ along the ideal $<s/r, y_1>$.
Hence, the projection morphism $X_{r, s/r}\to  \Proj R[y_1, y_2]$ is the blow-up along the closed subscheme $$\mathcal{Z}(<r, y_2>)\cup \mathcal{Z}(<s/r, y_1>).$$

\begin{equation*}
\begin{tikzpicture}
[back line/.style={densely dotted},cross line/.style={loosely dotted}]
\matrix (m) [matrix of math nodes, row sep=2 em,column sep=2.em, text height=1.5ex, text depth=0.50ex]
{
            & X_{r, s/r} &                        & Bl_{\mathcal{Z}(<r, y_2>)\cup \mathcal{Z}(<s/r, y_1>)} \Proj R[y_1, y_2]\\
 Sp(\mathcal{H}) & X_r    & \mathbb{P}^1_U\times_U \Proj R[y_1, y_2]  & \Proj R[y_1, y_2]\\
            & \mathbb{P}^1_U & \\
  };
\draw [right hook->, font=\scriptsize] (m-2-2) -- node[above] {$i$} (m-2-3);
\path[->,font=\scriptsize]
  (m-1-2) edge node [auto] {$\simeq, \theta$} (m-1-4)
   edge node [auto] {$\phi$} (m-2-2)
(m-1-4) edge node [auto] {$\eta$} (m-2-4)
 (m-2-1) edge node [auto] {$h'_{\mathcal{H}}$} (m-2-2)
 edge node [auto] {$h''_{\mathcal{H}}$} (m-1-2)
  (m-2-2) 
    edge node [auto] {} (m-3-2)
  (m-2-3) edge node [auto] {$pr_2$} (m-2-4)
   
;
\end{tikzpicture}
\end{equation*}
We apply~\cite[Proposition 3.7, Remark 3.8]{bs19} to the $n$-ghost homotopy given by the composition $\theta\circ h''_{\mathcal{H}} $. Since $\eta\circ \theta\circ h''_{\mathcal{H}}(0)$ avoids the closed subscheme $\mathcal{Z}(<r, y_2>)\cup \mathcal{Z}(<s/r, y_1>)$, the image of the $n$-ghost homotopy $\eta\circ \theta\circ h''_{\mathcal{H}}$ avoids the closed subscheme $\mathcal{Z}(<r, y_2>)\cup \mathcal{Z}(<s/r, y_1>)$. \\
  Thus, the restriction of $\eta\circ \theta\circ h''_{\mathcal{H}}$ to the closed subscheme $U(<r, s/r>)$ factors through 
\begin{center}
$\pone_{U(<r, s/r>)}-(\{(0:1), (1:0)\})\times U(<r, s/r>)\simeq \mathbb{G}_m\times U(<r, s/r>).$
\end{center}
Since $\mathbb{G}_m$ is $\A^1$-rigid, the restriction of $\eta\circ \theta\circ h''_{\mathcal{H}}$ to $\mathbb{G}_m\times U(<r, s/r>)$ is constant. In particular, $$\eta\circ \theta\circ h''_{\mathcal{H}}(0)=\eta\circ \theta\circ h''_{\mathcal{H}}(1).$$ Note that $\eta\circ \theta\circ h''_{\mathcal{H}}(0)$ is given by $y_1/y_2\mapsto 1$ and $\eta\circ \theta\circ h''_{\mathcal{H}}(1)$ is given by $y_1/y_2\mapsto r_2/r_1$. In other words, modulo every prime ideal $\mathfrak{p} $ containing  $<r, s/r>$,  $r_2/r_1=1$ in $R$
i.e. $r_2/r_1-1\in rad(<r, s/r>)$.\\ 
Thus, we have showed that, if two sections $\alpha_1, \alpha_2:U\to X_{\gamma}$ are $n$-ghost homotopic for some $n>0$, then 
$$
\dfrac{r_2}{r_1}-1\in rad<r_1, s/r_1>
$$
for all $s\in J$.\\
Conversely, we will show that if the two sections $\alpha_1, \alpha_2:U\to X_{\gamma}$ are such that
$$
\dfrac{r_2}{r_1}-1\in rad<r_1, s/r_1>
$$
for all $s\in J$, then $\alpha_1$ and $ \alpha_2$ are 1-ghost homotopic.\\
Under the above algebraic condition, we construct an explicit 1-ghost homotopy from $\alpha_1$ to $\alpha_2$. For this, consider the following Zariski open cover $V:=V_1\coprod V_2$ of $\A^1_U=\Spec R[S]$, where
\begin{center}
$V_1:=\A^1_U-\mathcal{Z}(<(J:<r>), 1+\delta S>)$\\
$V_2:=\A^1_U-\mathcal{Z}(<r>)$.
\end{center}
Here $(J:<r>)$ denotes the ideal $\{x\in R : rx\in J\}$.
Indeed, if $\mathfrak{p}\notin V_1$ i.e. $\mathfrak{p}$ is a prime ideal in $R[S]$ such that $\mathfrak{p}\supseteq <(J:<r>), 1+\delta S>$, then $\delta$ is a unit modulo  $\mathfrak{p}$. But $\delta\in rad<r, s/r>$ for all $s\in J$, hence  $\mathfrak{p}\nsupseteq <r, (J:<r>)>$. Thus $r\notin  \mathfrak{p}$ i.e. $\mathfrak{p}\in V_2$.\\
We define $h_i:V_i\to \Proj R[y_1, y_2]$ for $i=1, 2$ and take $h=h_1\coprod h_2 : V\to  \Proj R[y_1, y_2]$.  The morphism $h_1: V_1\to\Proj R[y_1, y_2]$ is given by $y_1/y_2\mapsto 1+\delta S.$ The morphism  $h_1$ lifts to a morphism $V_1\to X_{r}$ by using~(\ref{blXr}), since $h_1^*(<r, y_2>)=<1>$. Furthermore, the morphism $h_1: V_1\to X_r$ lifts to  $X_{0}:=Bl_{<(J:<r>), y_1>} X_{r}$, the blow up of $X_r$ along the ideal $<(J:<r>), y_1>$.   
Indeed, $$h_1^*(<(J:<r>), y_1>)=<(J:<r>), 1+\delta S>$$ is locally principal on $V_1$.
The morphism $h_2 : V_2 \to\Proj R[y_1, y_2]$ is given by $y_1/y_2\mapsto 1.$ By the similar argument as before, $h_2$ lifts to $X_{0}$.
We next construct a 1-ghost homotopy to   $X_{0}$ and compose this to the canonical morphism $X_{0}\to X_{\gamma}$ to get the desired claim.\\
The morphisms $\sigma_0, \sigma_1:U\to \A^1_U$  factor through $V_1\hookrightarrow \A^1_U$. The lifts $\widetilde{\sigma_0}, \widetilde{\sigma_1}$ of $\sigma_0, \sigma_1$ to $V$ are given by $U\to V_1\to V$, respectively. Write $W:=V\times_{\A^1_U} V=\coprod_{i,j=1,2} V_{ij}$, where $V_{ij}:=V_i\cap V_j.$ We will define a homotopy $\A^1\times W\to X_{0}$ on each component of $W$.
For $i=j$, we take the constant homotopy on $V_i\cap V_i=V_i.$ On the component $V_{12}$, we define $\Spec R[T]\times V_{12} \to \Proj R[y_1, y_2]$ as $y_2/y_1\mapsto (1+\delta S)^{-1}(1-T)+T$. On the component $V_{21}$, we define $\Spec R[T]\times V_{21} \to \Proj R[y_1, y_2]$ as $y_2/y_1\mapsto (1-T)+(1+\delta S)^{-1}T$.\\
Thus, we get an $\A^1$-homotopy connecting the two morphisms $$ W=V\times_{\A^1_U} V \rightrightarrows V\xrightarrow{h} X_{0}.$$ Hence, we get a 1-ghost homotopy connecting the sections $\alpha''_1, \alpha''_2: U\to X_{0}.$ 
We have the commutative diagram
$$\xymatrix{
                                  &    & X_0=Bl_{<(J:<r>), y_1>} X_{r}\ar[rd] \ar[ld] &\\
           Sp(\calH) \ar@/^/[urr]\ar@{-->}[r]  & 	X_{\gamma}\ar[rd] &	& X_r=Bl_{<r, y_2>} \Proj R[y_1, y_2] \ar[r]\ar[ld] & \Proj R[y_1, y_2]\\
	&  & \pone_U & &
}$$
Composing with the morphism $X_{0}\to X_{\gamma},$ we get a 1-ghost homotopy between the sections $\alpha_1, \alpha_2: U\to X_{\gamma}.$ Since $\alpha_1$ and $\alpha_2$ are connected by this explicit 1-ghost homotopy, it follows from~\cite[Lemma 4.1]{bhs} that, $\alpha_1$ and $\alpha_2$ map to the same element in $\pi_0^{\A^1}(X_{\gamma})(U)$.
\end{proof}
To summarize, Lemma~\ref{lemmaA} and Lemma~\ref{lemmaB} together give a complete classification of $n$-ghost homotopies of sections of $\pone_U$ that lift to the blow up $X_{\gamma}$. 
  
Now we continue the series of lemmas for the remaining cases, depending on codimension of the images in $B$ of the generic point $\eta$ and the closed point $u$ of $U$ under the morphism $\gamma:U\to B$.
\begin{lemma}
\label{lemma3}
Let $\alpha_1$, $\alpha_2$ be sections of $\widetilde{X}$ over $U$ which are connected by an $n$-ghost homotopy $h_{\mathcal{H}}:Sp({\mathcal{H}})\to \widetilde{X}$ for some $n>0$. We further assume that, there is a morphism $\gamma:U\to B$ such that the $n$-ghost homotopy lies over the morphism $\gamma$ in the sense of the Definition~\ref{liesovergamma}. Assume that the morphism $\gamma:U\to B$ satisfies either of the following properties. 
\begin{enumerate}[label=\ref{lemma3}(\arabic*)]
\item\label{lemma31} $\gamma(\eta)$ is the generic point of $B$ and $\gamma(u)=y$ is of codimension 1 in $B$. Let $Y:=\overline{\{y\}}\neq W$.
\item\label{lemma32} $\gamma(\eta)=y$ is of codimension 1 in $B$ and $\gamma(u)$ is a closed point of $Y=\overline{\{y\}}$, and $Y\neq W$.

\end{enumerate}
Then the sections  $\alpha_1$, $\alpha_2$ are 1-ghost homotopic and map to the same element in  $\pi_0^{\A^1}(\widetilde{X})(U)$.
\end{lemma}
\begin{proof}
First we consider the case~\ref{lemma31}. In this case $\gamma(\eta)$ is the generic point of $B$ and $\gamma(u)=y$ is of codimension 1 in $B$.  Let $Y=\overline{\{y\}}$ be the closure of $y$ with the reduced closed subscheme structure. Thus, $\alpha_1, \alpha_2$ factor through the $\pone$-fibration given by 
$\pi^{-1}(Y) \to Y$. 

 Furthermore, by assumption we have that $Y\neq W$, hence $\gamma$ factors through $B-W$, so we can replace $Y$ by $Y-W$. 
 In particular $\gamma(u)\in \Spec A$ for some affine open subscheme $\Spec A$ of $B-W$. Hence the morphism
 $\gamma$ factors though $\Spec A$, since $U$ is a local affine scheme. We can further choose $\Spec A$ such that the $\pone$-fibration $\pi: X\to B$, when restricted to $\Spec A$ is trivial \ie isomorphic to $\pone\x \Spec A$. In particular, the sections $\alpha_1, \alpha_2: U\to\~X\to X$ factor through 
 $\pone\times \Spec A$. Hence,  $\alpha_1, \alpha_2$ are $\A^1$-homotopic by the Remark~\ref{rmk1}.   
 
 Now we consider the case~\ref{lemma32}. In this case $\gamma(\eta)=y$ is of codimension 1 in $B$ and $\gamma(u)$ is a closed point of $Y=\overline{\{y\}}$, the closure of $y$ in $B$ with the reduced closed subscheme structure.
Thus, 
$\gamma$ factors through the $\pone$-fibration given by 
$\pi^{-1}(Y) \to Y$. 
By assumption, $Y\neq W$, so either $\gamma(u)\notin W$ or $\gamma(u)\in W$. 
We give separate arguments for these two cases. 

First assume that $\gamma(u)\notin Y\cap W$. In the case $\gamma(u)\notin Y\cap W$: replace $Y$ by $Y-W$ and again by the Remark~\ref{rmk1}, $\alpha_1, \alpha_2$ are $\A^1$-chain homotopic.

In the second case assume that $\gamma(u)\in Y\cap W$. In this case, let $p:=\gamma(u)\in Y\cap W$. We can also assume that $Y\cap W=\{\gamma(u)\}$, for if $Y\cap W=\{p_1,\cdots, p_m, p\}$ be the collection of distinct closed points, then $\gamma$ factors through $Y-\{p_1,\cdots, p_m\}$, and replace $Y$ by $Y-\{p_1,\cdots, p_m\}$. 
 
 Let $\overline{Y}$ be the normalization of $Y$, which is a smooth curve. Since $\gamma:U\to Y$ is dominant, it factors through $\mu:\overline{Y}\to Y$. Thus, the $n$-ghost homotopy and in particular, $\alpha_1, \alpha_2$ factor through the $\pone$-fibration over $\overline{Y}$ given by the pullback $X\times_B \overline{Y}$. 
Over $U$, the $n$-ghost homotopy to $X\times_B \overline{Y}$ over $U$ lying over $\gamma$ factors through $X\times_B \overline{Y}\times _{\overline{Y}} U=X\times_B U\simeq \pone_U=\Proj R[x_0, x_1]$, which lifts to the blow up along the closed subscheme $Z\times_B U$. 

Let us denote the blow-up of $\pone_U$ along the the closed subscheme $Z\times_B U$ by $X_{\gamma}$. Let $J\subset R$ denote the ideal $I_{\gamma^{-1}(W)}$, associated to the closed subscheme $\gamma^{-1}(W)\subseteq U$. Then the ideal in $R[x_0, x_1]$, associated to $Z\times_B U$ is given by $I_{\gamma}:=<J, x_0>$, where $<J, x_0>$ denotes the homogeneous ideal in $R[x_0, x_1]$ generated by $J$ and $x_0$. Then the discussion before, together with Lemma~\ref{lemmaA} and Lemma~\ref{lemmaB} show that $\alpha_1$ and $\alpha_2$ are connected by an explicit 1-ghost homotopy. Now composing the 1-ghost homotopy with the map $X_{\gamma}\to \~X$, we get the 1-ghost homotopy from $\alpha_1$ to $\alpha_2$. It follows from~\cite[Lemma 4.1]{bhs} that, $\alpha_1$ and $\alpha_2$ map to the same element in $\pi_0^{\A^1}(\~X)(U)$. This finishes the proof of the lemma.
 \end{proof}

\begin{lemma}
\label{lemma4}
Let $\alpha_1$, $\alpha_2$ be sections of $\widetilde{X}$ over $U$ which are connected by an $n$-ghost homotopy $h_{\mathcal{H}}:Sp({\mathcal{H}})\to \widetilde{X}$ for some $n>0$. We further assume that, there is a morphism $\gamma:U\to B$ such that the $n$-ghost homotopy lies over the morphism $\gamma$ in the sense of the Definition~\ref{liesovergamma}. Assume that the morphism $\gamma:U\to B$ satisfies either of the following properties. 
\begin{enumerate}[label=\ref{lemma4}(\arabic*)]
\item\label{lemma41}
 $\gamma(\eta)=y$ is of codimension 1 in $B$ and $\gamma(u)$ is a closed point of $Y=\overline{\{y\}}$, $Y=W$. 

\item\label{lemma42} $\gamma(\eta)$ is the generic point of $B$ and $\gamma(u)=y$ is of codimension 1 in $B$. Let $Y=\overline{\{y\}}$, $Y=W$. 

\item \label{lemma43} $\gamma(\eta)$ is the generic point of $B$ and $\gamma(u)$ is a closed point of $W$.
\end{enumerate}
Then the sections  $\alpha_1$, $\alpha_2$ are 1-ghost homotopic and map to the same element in  $\pi_0^{\A^1}(\widetilde{X})(U)$.
\end{lemma}
\begin{proof}
 We give a uniform argument for all the above cases. \\
  In the case~\ref{lemma41}: let $\gamma(u)=w$ be a closed point of $W$. Then the morphism $\gamma:U\to B$ factors through the canonical morphism $\Spec\mathcal{O}^h_{W, w}\to W\to B$. Consider the ideal $I_{\gamma}$ given by $<r_0, x_0>$, where $r_0\in R$ is the image of a uniformizer of the local ring $\mathcal{O}^h_{W, w}$, under the induced morphism $\mathcal{O}^h_{W, w}\to R$.\\
  In the case~\ref{lemma42}:  let $w$ be the generic point of $W$. Then, the morphism $\gamma:U\to B$ factors through the canonical morphism $\Spec\mathcal{O}^h_{B, w}\to B$. Consider the ideal $I_{\gamma}$ given by $<r_0, x_0>$, where $r_0\in R$ is the image of the uniformizer of the local ring $\mathcal{O}^h_{B, w}$, under the induced morphism $\mathcal{O}^h_{B, w}\to R$.\\
 In the case~\ref{lemma43}: let $\gamma(u)=w$ be a closed point of $W$. Then the morphism $\gamma:U\to B$ factors through the canonical morphism $\Spec\mathcal{O}^h_{B, w}\to B$.
Since, $W$ is a smooth closed subscheme of the smooth scheme $B$, we can find local parameters $t_1, t_2$ of $B$ at $w$ and a Zariski neighbourhood $V$ of $w$ in $B$ such that the ideal of $W$ in $V$ is given by $<t_1>.$ Then, the ideal $I_{\gamma}$ is given by $<r_0, x_0>$ where $r_0\in R$ is the image of $t_1$ in the local ring $\mathcal{O}^h_{B, w}$, under the induced morphism $\mathcal{O}^h_{B, w}\to R$. \\  
 By \cite[Lemma 2.12]{bs19}, since $B$ is $\A^1$-rigid, an $n$-ghost homotopy joining $\alpha_1, \alpha_2: U\to \widetilde{X}$ lying over $\gamma$ factors through $X\times_{\gamma, B} U \to X$. Since $U$ is a Henselian local scheme, the $\pone$-fibration    $X\times_{\gamma, B} U\to U$ is trivial i.e. isomorphic to $\pone_U$. 
  So we have sections $\alpha_1, \alpha_2: U\to \pone_U$ which lift to the blow up $X_{\gamma}$ of  $\pone_U$ along the ideal sheaf $I_{\gamma}$ such that these are $n$-ghost homotopic on $X_{\gamma}$ for some $n>0$.   In all the cases~\ref{lemma41},~\ref{lemma42} and~\ref{lemma43}, write $I_{\gamma}=<J, x_0>$ for $J=<r_0>$.\\
Then the discussion before, together with Lemma~\ref{lemmaA} and Lemma~\ref{lemmaB} show that $\alpha_1$ and $\alpha_2$ are connected by an explicit 1-ghost homotopy. Now composing the 1-ghost homotopy with the map $X_{\gamma}\to \~X$, we get the 1-ghost homotopy from $\alpha_1$ to $\alpha_2$. It follows from~\cite[Lemma 4.1]{bhs} that, $\alpha_1$ and $\alpha_2$ map to the same element in $\pi_0^{\A^1}(\~X)(U)$. This finishes the proof of the lemma.
\end{proof}

As a consequence of the above lemmas, we conclude the proof of Proposition~\ref{n-ghostlyingovergamma}. For the sake of completeness, at the cost of some repetition we present the proof here.
%
\begin{proof}[Proof of Proposition~\ref{n-ghostlyingovergamma}]
We make various cases, depending on the  \textbf{codimension of the images in $B$ of the generic point $\eta$ and the closed point $u$ of $U$ under the morphism $\gamma:U\to B$}.
We list the cases we will consider:
\begin{enumerate}
\item\label{1} $\gamma(u)$ is the generic point of $B$.
\item\label{2} $\gamma(\eta)$ is a closed point $b$ of $B$.
\item\label{3} $\gamma(u)=\gamma(\eta)=y$ of codimension 1 in $B.$
%
\item \label{4} $\gamma(\eta)=y$ of codimension 1 in $B$ and $\gamma(u)$ is a closed point of $Y=\overline{\{y\}}.$
\item\label{5} $\gamma(\eta)$ is the generic point of $B$ and $\gamma(u)=y$ of codimension 1 in $B$. Let $Y=\overline{\{y\}}$.
\item\label{6} $\gamma(\eta)$ is the generic point of $B$ and $\gamma(u)$ is a closed point of $W$.
\end{enumerate}
Cases~\ref{1} and~\ref{2} follow from the Lemma~\ref{lemma1}. Case~\ref{3} follows from the Lemma~\ref{lemma2}. Case~\ref{4} follows from the Lemma~\ref{lemma3}(2) and~\ref{lemma4}(1). Case~\ref{5} follows from the Lemma~\ref{lemma3}(1) and Lemma~\ref{lemma4}(2). Case~\ref{6} follows from Lemma~\ref{lemma4}(3).

\end{proof}
\end{section}

\paragraph{Acknowledgement} The author is supported by the Postdoctoral Fellow position in Indian Institute of Science Education and Research (IISER), Mohali. The author would like to thank Chetan Balwe for continuous suggestions and comments that immensely improved the article.

\end{document}